\documentclass[11pt]{amsart}
\usepackage{amsmath}
\usepackage{amssymb}
\usepackage{amsthm}
\usepackage{latexsym}
\usepackage{hyperref}
\usepackage{enumerate}
\usepackage[all]{xy}

\usepackage{amssymb,amsmath,amscd,url,geometry,pdflscape}

\newtheorem{thm}{Theorem}[section]

\def\LLC{\mathrm{LLC}}

\def\fo{\mathfrak{o}}
\def\fp{\mathfrak{p}}

\def\Gal{\mathrm{Gal}}
\def\Hom{\mathrm{Hom}}

\def\dep{\mathrm{dep}}

\def\val{\mathrm{val}}

\def\F{\mathbb{F}}

\def\C{\mathbb{C}}

\def\Z{\mathbb{Z}}

\def\Q{\mathbb{Q}}
\def\SL{\mathrm{SL}}

\def\bbG{\mathbb{G}}
\def\W{\mathbf{W}}

\begin{document}

\title[Depth]{On depth in the local Langlands correspondence for tori}

\author[A.-M. Aubert]{Anne-Marie Aubert}
\address{CNRS, Sorbonne Universit\'e, Universit\'e Paris Diderot,
Institut de Math\'ematiques de Jussieu -- Paris Rive Gauche, IMJ-PRG, 
F-75005 Paris, France}
\email{anne-marie.aubert@imj-prg.fr}
\author[R. Plymen]{Roger Plymen}
\address{School of Mathematics, Southampton University, Southampton SO17 1BJ,  England
\emph{and} School of Mathematics, Manchester University, Manchester M13 9PL, England}
\email{r.j.plymen@soton.ac.uk \quad roger.j.plymen@manchester.ac.uk}

\date{\today}
\maketitle

\begin{abstract}  
Let $K$ be a  non-archimedean local field.    In the local Langlands correspondence for tori over $K$, we prove an asymptotic result for the depths.
\end{abstract}

\tableofcontents

\section{Introduction}  Let $K$ be a non-archimedean local field.   The LLC (Local Langlands Correspondence) for tori induces an isomorphism
\[
\lambda_T : \Hom(T(K), \C^\times) \simeq \mathrm{H}^1(\W_K, T^\vee)
\]
where $\W_K$ is the Weil group of $K$ , $T$ is a torus defined over $K$ and $T^\vee=X^*(T)\otimes_\Z\C^\times$ is the complex dual torus of $T$, see \cite{La} and \cite{Yu}.   

The Moy-Prasad theory associates to each character $\chi$ of $T(K)$  an invariant $\dep(\chi)$ called the depth of $\chi$. Also for each $\lambda\in \mathrm{H}^1(\W_K, T^\vee)$, one defines the notion of depth 
$\dep(\lambda)$. For background material on depth, see \cite{ABPS}. 
    
Concerning depth-preservation by the LLC, we have the theorem of Yu \cite[\S 7.10]{Yu}:  In the LLC for tori, if $T$ splits over a tamely ramified extension, then we have 
 \[
 \dep(\chi) = \dep (\lambda_T(\chi)).
 \] 
 
Our main result is as follows.
 
 \begin{thm} \label{main}     Let $T= R_{L/K} \mathbb{G}_m$ be an induced
torus, where $L/K$ is a finite and Galois extension of non-archimedean local fields. In the $\LLC$ for $T$, depth is preserved (for positive depth characters)  if and only if $L/K$ is tamely ramified.   
 
 If $L/K$ is wildly ramified then, for each  character $\chi$ of $T(K)$ such that $\dep(\chi)>0$,
 we have
\[
\dep(\lambda_T(\chi) ) > \dep(\chi), 
\]
and 
\[
\dep(\lambda_T(\chi)) / \dep(\chi) \to 1 \quad \textrm{as} \quad \dep(\chi) \to \infty.   
\]  
\end{thm}

We illustrate this result in \S \ref{EXX} with several examples of wildly ramified extensions.   
 
\section{The notion of depth}   

We recall from \cite{MoPr} that $T(K)$ carries a Moy-Prasad filtration $\{T(K)_r : r \geq 0\}$. When $T=R_{L/K}(\bbG_m)$, we have
\[T(K)_r=\{t\in T(K) :  \val_L(t-1)\ge er\},\]
where $\val_L$ is the valuation on $L$ normalized so that $\val_L(L^\times)=\Z$ and $e=e(L/K)$ is the ramification index of $L/K$.

The \emph{depth} of a character $\chi : T(K) \to \C^\times$ is 
\[
\dep(\chi):=\inf \{ r \geq 0 : T(K)_s \subset \ker(\chi) \:  \textrm{for}  \:  s > r\}.
\]

The Weil group $\W_K$ carries an upper number filtration $\{\W_K^r : r \geq 0\}$ and the depth of a parameter $\lambda \in \mathrm{H}^1(\W_K, T^\vee)$ is defined to be
\[
\dep(\lambda):=\inf \{ r \geq 0 : \W_K^s  \subset \ker(\lambda) \:  \textrm{for}  \:  s > r\}.
\]

\section{Asymptotics} We review the ramification of the Galois group $G(L/K)$ in the lower numbering, see \cite[IV. \S 1]{Ser}.  We have a decreasing sequence of normal subgroups of $G$: 
\begin{eqnarray}\label{Gram}
G = G_{-1} \supseteq G_0 \supseteq G_1 \supseteq G_2 \supseteq G_3 \supseteq \cdots \supseteq G_i = G_{i+1} = \cdots = \{1\}
\end{eqnarray}
where $G_i$ is the group of elements in $G$ acting trivially on the quotient of the ring of integers of $L$ by the $(i+1)$th power of its maximal ideal.

Now $G_0$ is the inertia group and its fixed field $L^{G_0}$ is the maximal unramified extension $K_0$ of $K$ in $L$.     By the fundamental theorem of Galois theory, we have 
$[L:K_0] = [L:L^{G_0}] = |G_0|$ and so the ramification degree $e(L/K)$ is the order of $G_0$:
\begin{eqnarray}\label{eG}
e(L/K)  =  |G_0|.
\end{eqnarray}
So,  $L/K$ is unramified if and only if $G_0 = 1$.  In that case, the largest ramification break in the sequence (\ref{Gram}) occurs at $-1$.

 Let us write the ramification index as $e = p^n m$ with $(m,p) = 1$ and $p = char(K)$, then $| G_1| = p^n$.
Now  $L/K$ is tamely ramified if and only if $G_1 = 1$.  In that case, the largest ramification break in the sequence (\ref{Gram}) occurs at  $0$.  
Moreover, if $K_1$ is the maximal tamely ramified extension of $K$ inside $L$, then $K_1$ is the fixed field of $G_1$, and is hence an extension of $K_0$, and has Galois group $G/G_1$.

 We thus have a tower of fields $L \supset K_1 \supset K_0 \supset K$
where $K_0/K$ is unramified of degree $| G/G_0 |$, the extension $K_1/K_0$ is totally tamely ramified of degree 
$| G_0/G_1 |$ 
and $L/K_1$ is totally wildly ramified of degree $| G_1 |$.

The group $G_0/G_1$ is cyclic, and its order is prime to the characteristic of the residue field $\overline{L}$; if the characteristic of $\overline{L}$ is $p \neq 0$, 
the quotients $G_i/ G_{i+1}, \; i \geq 1$, are abelian groups, 
and are direct products of cyclic groups of order $p$; the group $G_1$ is a $p$-group. See \cite[IV. \S 2]{Ser}.   

If $t$ is a real number $\geq -1$, $G_t$ denotes the ramification group $G_i$, where $i$ is the smallest integer $\geq t$.   Then the Hasse-Herbrand function \cite[IV. \S 3]{Ser} is
\[
\varphi_{L/K}(u): = \int_0^u\frac{1}{(G_0:G_t)}dt.
\]  

\begin{thm}\label{Herbrand}   Let $L/K$ be a finite and Galois extension of  non-archimedean local fields.   Let $e = e(L/K)$ denote the ramification index of $L/K$, let $r > 0$.      Then we have
\[
\varphi_{L/K}(er) = r
\]
if and only if $L/K$ is tamely ramified.   If $L/K$ is wildly ramified then
\[
\varphi_{L/K}(er) > r
\]
and
\[
\frac{\varphi_{L/K}(er)}{r} \to 1 \quad \textrm{as} \quad r \to \infty.
\]
\end{thm}

\begin{proof}  Let $r>0$ and let $b$ denote the largest ramification break in the Galois group of $L/K$, and let $0 < x \leq  b$.   
We have 
\begin{eqnarray*}
\varphi_{L/K}(x) - \frac{x}{e} & = & \int_0^x\frac{1}{(G_0: G_t)} dt  -  \int_0^x\frac{1}{(G_0: \{1\})} dt\\
& = & \int_0^x \left( \frac{1}{(G_0: G_t)}   -  \frac{1}{(G_0:\{1\})} dt  \right )\\
& > & 0 
\end{eqnarray*}
so that
\[
er \leq b \implies \frac{\varphi_{L/K}(er)}{r} > 1.
\]

Let $b$ denote the largest ramification break in the Galois group of $L/K$, and let $x \geq  b$.    We have

\begin{eqnarray*}
\varphi_{L/K}(x) & = & \int_0^b \frac{1}{(G_0:G_t)} dt + \int_b^x \frac{1}{(G_0:G_t)} dt  \\
& = & \varphi_{L/K}(b) + \int_b^x \frac{1}{(G_0:G_t)} dt \\
& = & \varphi_{L/K}(b) + \int_b^x \frac{1}{(G_0: \{1\})}\\ 
& = & \varphi_{L/K}(b) + \int_b^x \frac{1}{|G_0|} dt \\
& = & \varphi_{L/K}(b) + \frac{x-b}{|G_0|}\\
& = & \varphi_{L/K}(b) + \frac{x-b}{e(L/K)}
\end{eqnarray*}
by (\ref{eG}).   We then have, with $er \geq b$: 
\begin{eqnarray*}
\varphi_{L/K}(er) & = & \varphi(b) + \frac{er - b}{e} \\
& = & r + \varphi_{L/K}(b) - \frac{b}{e}
\end{eqnarray*}
and so we have
\begin{eqnarray*}
\frac{\varphi_{L/K}(er)}{r}  & = & 1 + \frac{\varphi(b)}{r} - \frac{b}{er}.
\end{eqnarray*}

Introduce the following invariant of the field extension $L/K$:
\[
\mathfrak{a}(L/K) : = \varphi _{L/K}(b) - \frac{b } { e}.
\]
Then we have 
\[
\frac{\varphi_{L/K}(er)}{r} = 1 + \frac{\mathfrak{a}(L/K)}{r}
\]
and
so that we have
\[
\frac{\varphi_{L/K}(er)}{r} \to 1 \quad \textrm{as}  \quad r \to \infty .
\]

\medskip

Let the ramification breaks in the lower numbering occur at $b_1, b_2, \ldots, b_k = b$.   We have 
\begin{eqnarray*}
\frac{b}{e} & = & \int_0^b\frac{1}{(G_0: \{1\})} dt \\
& = & \int_0^{b_1} \frac{1}{(G_0:\{1\})} dt + \cdots +  \int_{b_{k-1}}^b \frac{1}{(G_0:\{1\})} dt \\
& \leq  & \int_0^{b_1} \frac{1}{(G_0: G_{b_1})} dt + \cdots +  \int_{b_{k-1}}^b \frac{1}{(G_0: G_b)} dt \\
& = & \varphi_{L/K}(b)
\end{eqnarray*}
so that
\[
\mathfrak{a}(L/K) = \varphi_{L/K}(b) - \frac{b}{e} \geq 0
\]
with equality  if and only if $b = 0$, i.e. if and only if $L/K$ is tamely ramified.   

Now $L/K$ is tamely ramified if and only if  $G_1 = 1$.   In that case, we have
\begin{eqnarray*}
\varphi_{L/K}(x) & = & \int_0^x \frac{1}{(G_0:G_t)}\\
& = & \int_0^x\frac{1}{(G_0: \{1\})}\\
& = & \int_0^x \frac{1}{|G_0|}\\
& = & \frac{x}{e}
\end{eqnarray*}
by (\ref{eG}) and so we have
\[
\varphi_{L/K}(er) = r
\]
for all $r \geq 0$.   
\end{proof}  

\section{Proof of Theorem~\ref{main}}    We have the elegant recent formula of Mishra and Patanayak \cite{MP}:
\begin{equation}\label{mp}
\varphi_{L/K}(e \cdot \dep(\chi))  = \dep(\lambda_T(\chi))
\end{equation}
where  
 $\chi$ is any  character of $T(K)$.
Now choose a character $\chi$ of positive depth and set $r = \dep(\chi)$ in Theorem~\ref{Herbrand}.

\section{Examples of wildly ramified extensions $L/K$}\label{EXX} 

\subsection{Characteristic $p$}  Let $K$ be a local field of characteristic $p$.   
Let $\mathfrak{o}$ be the ring of integers in $K$ and $\mathfrak{p}\subset\mathfrak{o}$ the maximal ideal. 
Let $\wp(x) = x^p - x$.

Let $\overline{K} = K/\wp(K)$.   Let $D \neq \overline{\fo}$ be an $\F_p$-line in $\overline{K}$, $m$ the integer such that $D \subset \overline{{\fp}^{-m}}$ but 
$D\not\subset \overline{{\fp}^{-m+1}}$; we know that $m$ is $>0$ and prime to $p$.   Fix an element $a \in \fp^{-m}$ whose image generates $D$, let $\alpha$ be a root of $X^p - X - a$ (in an algebraic closure of 
$K$), and let $L = K(\alpha) = K(\wp^{-1}(D))$.  This example is due to Dalawat, see \cite[\S 6]{Da}; but see also \cite[IV, \S 2, Exercise 5]{Ser} and references therein.  

The extension $L/K$ is totally (and wildly) ramified.   The unique ramification break of the degree $p$ cyclic extension $L/K$ occurs at $m$, 
\[G = G_0 = G_1 = \cdots = G_m, 
\; G_{m+1} = \cdots = \{1\}
\]
 see 
\cite[Proposition 14]{Da}.  Set $T = L^\times$, then $T$ is a wildly ramified
torus.   

Here, we have $e = p, b = m, \varphi_{L/K}(b) = m, \mathfrak{a}_{L/K} = m(1 - p^{-1})$.

\subsection{Characteristic $0$}  Let now $F$ be a non-archimedean local field of characteristic $0$ with  residue field $\mathbb{F}_q$ with $q = p^n, \; n \geq 1$.   Let $\beta$ be a root of the polynomial $X^q - X - \alpha$ with $\alpha \in K$, $\val_K(\alpha) >
- q e(F) /(q - 1)$.   Let $L = F(\beta)$, $T = L^\times$.   If $p \not | \;  \val_K(\alpha)$, $\val_K(\alpha) < 0$ then $L/K$ is a totally ramified 
Galois extension of degree $q$, and if $G = \Gal(L/K)$ then the unique ramification break of $L/K$ occurs at m:
\[G = G_0 = G_1 = \cdots = G_m, 
\; G_{m+1} = \cdots = \{1\}
\]
 where $m = -\val_K(\alpha)$.  This example is due to Abrashkin, see \cite[p. 79]{FV}.   

Here, we have $e = q, b = m, \varphi(b) = m, \mathfrak{a}_{L/K} =  m(1 - q^{-1})$.   

\subsection{Characteristic $2$}     Let $K$ be a local field of characteristic $2$, and let $L/K$ be a totally ramified quadratic extension: there are countably many of these, with ramification breaks given by 
$m = 1, 3, 5, 7, \ldots$, see \cite{AMPS}.       

Here, we have $e = 2, b = m, \varphi(b) = m, \mathfrak{a}_{L/K} = m(1 - 2^{-1}) = m/2$.     

\subsection{Example with $2$ breaks}  Let $K = \Q_2(\sqrt 5)$ and let $L/K$ be the totally ramified extension of local fields constructed in \cite[\S 4]{Ser2}.   The Galois group $\Gal(L/K)$ is 
the group 
$\{\pm 1, \pm \mathbf{i}, \pm \mathbf{j}, \pm \mathbf{k}\}$ of unit quaternions, and
\[
G_{-1} = G = G_0 = G_1 \supset G_2 = G_3 \supset \{1\}
\]
with $G_2 = \{\pm 1\}$. 

Here, we have 
\begin{eqnarray*}
|G_1| & = & 2^4\\
 G_1/G_2 & = & (\Z/2\Z)^3\\
 e &  = & 2^4\\
  b & = & 3\\
   \varphi(b) & = & 5/4\\
    \mathfrak{a}_{L/K} & = & 17/16
\end{eqnarray*}    

\subsection{Example in \cite{MP}}   Here, we have $L= \Q_p(\zeta_{p^n})$ with $n \geq 2, \;  K= \Q_p(\zeta_p)$.  Then $L/K$ is a wildly ramified extension with $e(L/K) = p^{n-1}$.   
From the calculations in \cite{MP}, we have
\begin{eqnarray*}
e & = & p^{n-1}\\
 b & = & p^{n-1} -1\\
  \varphi_{L/K}(b) & = & (n-1)(p-1)\\
   \mathfrak{a}_{L/K} & = & (n-1)(p-1) + 1 - p^{1-n}
\end{eqnarray*}

\end{document}